\documentclass[a4paper]{amsart}
\usepackage[all]{xy}

\usepackage[colorlinks=true]{hyperref}
\usepackage{enumerate}
\usepackage{amssymb}
\usepackage{comment}
\newtheorem{thm}{Theorem}[section]

\newtheorem{prop}[thm]{Proposition}
\newtheorem{lem}[thm]{Lemma}
\newtheorem{cor}[thm]{Corollary}

\theoremstyle{definition}
\newtheorem{defn}[thm]{Definition}

\newtheorem{exs}[thm]{Examples}

\theoremstyle{remark}
\newtheorem{rem}[thm]{Remark}

\newcommand{\Rr}{\mathbb R}

\newcommand{\Nn}{\mathbb N}

\renewcommand{\d}{\mathrm d}                           

\renewcommand{\d}{\mathrm d}               

\begin{document}
\title[Nijenhuis and compatible tensors on Lie and Courant algebroids]{Nijenhuis and compatible tensors on Lie and Courant algebroids}

\author{P. Antunes}
\address{CMUC, Department of Mathematics, University of Coimbra, 3001-454 Coimbra, Portugal}
\email{pantunes@mat.uc.pt}
\author{J. M. Nunes da Costa}
\address{CMUC, Department of Mathematics, University of Coimbra, 3001-454 Coimbra, Portugal}
\email{jmcosta@mat.uc.pt}

\begin{abstract}
We show that well known structures on Lie algebroids
can be viewed as Nijenhuis tensors or pairs of compatible tensors on Courant algebroids. We study compatibility and construct hierarchies of these structures.
\end{abstract}

\maketitle

\section*{Introduction}
\label{sec:introduction}

Pairs of tensor fields on manifolds, which are compatible in a certain sense, were studied by Magri and Morosi \cite{magrimorosi},
in view of their application to integrable hamiltonian systems. Besides Poisson-Nijenhuis manifolds -- manifolds equipped with a Poisson bivector
and a Nijenhuis $(1,1)$-tensor which are compatible in such a way that it is possible to define a hierarchy of Poisson-Nijenhuis structures on these manifolds,
 the work of Magri and Morosi also covers the study of $\Omega N$ and $P \Omega$ structures. These are pairs of tensors formed respectively, by a closed $2$-form and
  a Nijenhuis tensor ($\Omega N$) and a Poisson bivector and a closed $2$-form ($P \Omega$) satisfying suitable compatibility conditions. Another type of structure
  that can be considered on a manifold is a Hitchin pair. It is a pair formed by a symplectic form and a $(1,1)$-tensor that was introduced by Crainic
  \cite{crainic} in relation with generalized complex geometry. All
  these structures, defined by pairs of tensors, were studied in the Lie
   algebroid setting by Kosmann-Schwarzbach and Rubtsov \cite{YKSRubtsov} and by one of the authors
  \cite{antunes2}. Finally, we mention
   complementary forms on Lie algebroids, which were defined by Vaisman \cite{vaisman} and also considered in \cite{YKSRubtsov} and
   \cite{antunes2}, and that can be viewed as Poisson structures on the dual Lie algebroid.

The aim of  the present paper is to show that all the structures referred to above, although they have different nature on Lie algebroids, once carried over to
 Courant algebroids, are all of the same type: they are Nijenhuis tensors. In this way,
 we obtain a unified theory of Nijenhuis structures on Courant algebroids. In order to include Poisson quasi-Nijenhuis structures with background in this unified theory, we consider
 a stronger version of this notion, which we call {\em exact}
 Poisson quasi-Nijenhuis structure with background.
 This seems to be the natural definition, at least in this context.

We show that the structures  defined by pairs of tensors on a Lie algebroid can also be characterized using the  notion of compatible pair of tensors on a Courant
 algebroid, introduced in \cite{anutuneslaurentndc}.

 An important tool in this work is the Nijenhuis concomitant of two $(1,1)$-tensors on
a Courant algebroid. It was originally
 defined for manifolds by Nijenhuis in \cite{nijenhuis}  and then extended to the Courant algebroid framework in \cite{stienon} and in \cite{anutuneslaurentndc}.
We use the Nijenhuis concomitant to study the compatibility of structures from the usual point of view, i.e., we say that two structures
 of the same type are compatible if their sum is still a structure of the same type. Thus, we can talk about compatible Poisson-Nijenhuis, $\Omega N$ and
$P \Omega$ structures, as well as compatible complementary forms and compatible Hitchin pairs.

The extension to Lie algebroids of the Magri-Morosi hierarchies of Poisson-Nijenhuis structures on manifolds, was done in \cite{magriYKS}.
As it happens in the case of manifolds, the hierarchies on Lie algebroids are constructed through deformations by Nijenhuis tensors.
In this paper we construct similar hierarchies of $\Omega N$ and $P \Omega$ structures on Lie algebroids, and their deformations, and also hierarchies
of complementary forms.
 Elements of these hierarchies provide examples of compatible structures in the sense
 described above.

 Our computations widely use the {\em big bracket} -- the Poisson bracket induced by the symplectic structure on the cotangent bundle of a supermanifold.  The Courant algebroids that we shall consider in this paper are doubles $(A\oplus A^*,\Theta)$
  of
 protobialgebroids structures on $(A,A^*)$ \cite{YKS05}, in the simpler cases where $\Theta$ is a function that determines a Lie algebroid structure on $A$, or on $A^*$,
 sometimes in the presence of a background (a closed $3$-form on $A$).

The paper is organized as follows. Section 1 contains a short review of Courant and Lie algebroids in the supergeometric framework while, in section 2, we
recall the notion of Nijenhuis tensor on a Courant algebroid and of Nijenhuis concomitant of two tensors. In section 3, we characterize Poisson bivectors
 and closed $2$-forms on a Lie algebroid $(A,\mu)$ as
Nijenhuis tensors on the Courant algebroid $(A \oplus A^*, \mu)$. In section 4,
we show how Poisson-Nijenhuis,
$\Omega N$ and $P \Omega$ structures  and also Hitchin pairs on a Lie algebroid $(A,\mu)$ can be seen either as Nijenhuis tensors or compatible pairs
 of tensors
on the Courant algebroid $(A \oplus A^*, \mu)$.
Considering, in section 5, the Courant algebroid with background $(A \oplus A^*, \mu + H)$, we see exact Poisson
quasi-Nijenhuis structures with background as Nijenhuis tensors on this Courant algebroid, recovering a result in \cite{antunes1}.
For Poisson
quasi-Nijenhuis structures (without background) a special case where two $3$-forms involved are exact is also considered.
The case of complementary forms is  treated in section 6. Section 7
is devoted to the compatibility of structures on a Lie algebroid, defined by pairs of tensors.
Sections 8, 9 and 10 treat the problem of defining hierarchies of structures on Lie algebroids. We start by showing, in section 8, that when a pair of tensors defines a certain structure on a Lie algebroid,
the same pair of tensors defines a structure of the same kind for a whole hierarchy of deformed Lie algebroids. Then, in section 9, we construct hierarchies
 of structures defined by pairs of tensors and lastly, in section 10, we show that within one hierarchy, all the elements are pairwise compatible.

We recall that if one relaxes the Jacobi identity in the definition of a Lie (respectively, Courant) algebroid we obtain what is called a {\em pre-Lie}
 (respectively, {\em pre-Courant}) algebroid.
The proof of our results does not use the Jacobi identity of the bracket, whether if it is a Lie or a Courant algebroid bracket. Therefore, they also hold in the more general settings of pre-Lie and pre-Courant algebroids, respectively.

\section{Courant and Lie algebroids in supergeometric terms}  \label{subsection:1.1}
We begin this section by introducing the supergeometric setting, following the same approach as in \cite{voronov, roy}. Given a vector bundle $A \to M$, we denote by $A[n]$ the graded manifold obtained by shifting the fibre degree by $n$. The graded manifold $T^*[2]A[1]$ is equipped with a canonical symplectic structure which induces a Poisson bracket on its algebra of functions $\mathcal{F}:=C^\infty(T^*[2]A[1])$. This Poisson bracket is sometimes called the \emph{big bracket} (see \cite{YKS05}).

Let us describe locally the Poisson bracket of the algebra $\mathcal{F}$. Fix local coordinates $x^i, p_i,\xi^a, \theta_a$, $i \in \{1,\dots,n\}, a \in \{1,\dots,d\}$, in $T^*[2]A[1]$, where $x^i,\xi^a$ are local coordinates on $A[1]$ and $p_i, \theta_a$ are their associated moment coordinates. In these local coordinates, the Poisson bracket is given by
 $$ \{p_i,x^i\}=\{\theta_a,\xi^a\}=1,  \quad  i =1, \dots, n, \, \, a=1, \dots , d, $$
while all the remaining brackets vanish.

The Poisson algebra of functions $\mathcal{F}$ is endowed with a $(\Nn \times \Nn)$-valued bidegree. We define this bidegree locally but it is well defined globally (see \cite{voronov, roy} for more details). The bidegrees are locally set as follows: the coordinates on the base manifold $M$, $x^i$, $i \in \{1,\dots,n\}$, have bidegree $(0,0)$, while the coordinates on the fibres, $\xi^a$, $a \in \{1,\dots,d\}$, have bidegree $(0,1)$ and their associated moment coordinates, $p_i$ and $\theta_a$, have bidegrees $(1,1)$ and $(1,0)$, respectively.
The algebra of functions $\mathcal{F}$ inherits this bidegree and we set
$$   \mathcal{F}=\bigoplus_{(k,l) \in \Nn \times \Nn} \mathcal{F}^{k,l}, $$
where $\mathcal{F}^{k,l}$ is the $C^\infty(M)$-module of functions of bidegree $(k,l)$. The \emph{total degree}
 of a function $f \in \mathcal{F}^{k,l}$ is equal to $k+l$ and the subset of functions of total degree $t$ is noted $\mathcal{F}^t$.
We can verify that the big bracket has bidegree $(-1,-1)$, i.e.,
$$\{\mathcal{F}^{k_1,l_1},\mathcal{F}^{k_2,l_2}\}\subset \mathcal{F}^{k_1+k_2-1,l_1+l_2-1},$$
and consequently, its total degree is $-2$. Thus, the big bracket on functions of lowest degrees,
$\{\mathcal{F}^0,\mathcal{F}^0\}$ and $\{\mathcal{F}^0,\mathcal{F}^1\}$, vanish. For $\mathcal{X}, \mathcal{Y} \in \mathcal{F}^1= \Gamma(A \oplus A^*)$,
 $\{\mathcal{X}, \mathcal{Y}\}$ is an element of $\mathcal{F}^0=C^\infty(M)$  and is given by
$$\{\mathcal{X}, \mathcal{Y}\}=\langle\mathcal{X}, \mathcal{Y}\rangle,$$
where $\langle .,.\rangle$ is the canonical fiberwise symmetric bilinear form on $A\oplus A^*$.

\

Let us recall that a \emph{Courant} structure on a vector bundle $E \to M$ equipped with a fibrewise non-degenerate symmetric bilinear form $\langle.,.\rangle$ is a pair $(\rho, [.,.])$, where the \emph{anchor} $\rho$ is a bundle map from $E$ to $TM$ and the \emph{Dorfman bracket} $[.,.]$ is a $\mathbb{R}$-bilinear (not necessarily skew-symmetric) map on $\Gamma(E)$ satisfying
\begin{eqnarray}
& \rho(X)\cdot\langle Y,Z\rangle=\langle[X,Y],Z\rangle +  \langle Y,[X,Z]\rangle, \label{pre_Courant1}\\
&\rho(X)\cdot \langle Y,Z\rangle=\langle X, [Y,Z]+ [Z,Y]\rangle, \label{pre_Courant2}\\
& [X, [Y,Z]]=[[X,Y],Z]+ [Y,[X,Z]], \label{Jacobi_id}
\end{eqnarray}
 for all $X,Y,Z \in \Gamma(E)$.
From (\ref{pre_Courant1}) and (\ref{pre_Courant2}), we get \cite{YKS05}
$$[X, fY]=f[X,Y]+ (\rho(X).f)Y,$$ for all $X,Y \in \Gamma(E)$ and $f \in C^\infty(M)$.

In this paper we are only interested in exact Courant algebroids.
Although many of the properties and results we recall next hold in the general case, we shall consider the case where
the vector bundle $E$ is the Whitney sum of a vector bundle $A$ and its dual, i.e., $E=A \oplus A^*$,
and $\langle .,.\rangle$ is the canonical fiberwise symmetric bilinear form. So, from now on, all the Courant structures will be defined on
$(A \oplus A^*,\langle.,.\rangle)$.

From \cite{roy} we know that
there is a one-to-one correspondence between Courant structures on \mbox{$(A \oplus A^*,\langle.,.\rangle)$}
and functions $\Theta \in \mathcal{F}^3$ such that \mbox{$\{\Theta,\Theta\}=0$}.
The anchor and Dorfman bracket associated to a given $\Theta\in \mathcal{F}^3$ are defined, for all $\mathcal{X},\mathcal{Y} \in \Gamma(A \oplus A^*)$ and
$f \in C^\infty(M)$, by the derived bracket expressions
$$\rho(\mathcal{X})\cdot f=\{\{\mathcal{X},\Theta\},f\} \quad {\hbox{and}} \quad {[\mathcal{X},\mathcal{Y}]=\{\{\mathcal{X},\Theta\},\mathcal{Y}\}}.$$
For simplicity, we shall denote a Courant algebroid by the pair $(A \oplus A^*, \Theta)$ instead of the triple $(A \oplus A^*, \langle.,. \rangle, \Theta)$.

A Courant structure $\Theta \in \mathcal{F}^3$ can be decomposed using the bidegrees:
\begin{equation} \label{Theta}
\Theta=\mu + \gamma + \phi + \psi,
\end{equation}
with $\mu \in \mathcal{F}^{1,2}, \gamma \in \mathcal{F}^{2,1}, \phi \in \mathcal{F}^{0,3}=\Gamma(\wedge^3 A^*)$ and $\psi \in \mathcal{F}^{3,0}=\Gamma(\wedge^3 A)$.
We recall from~\cite{roy} that, when $\gamma = \phi = \psi =0$, $\Theta$ is a Courant structure on $A \oplus A^*$ if and only if $(A,\mu)$ is a Lie algebroid.
 The anchor and the bracket of the Lie algebroid are defined, respectively by
$$\rho(X)\cdot f=\{\{X,\mu \},f\} \quad {\hbox{and}} \quad {[X,Y]_\mu=\{\{X,\mu \},Y\}},$$
for all $X,Y \in \Gamma(A)$ and
$f \in C^\infty(M)$, while the Lie algebroid differential is given by
$$d_\mu= \{ \mu, . \}.$$

A function  $\Theta \in \mathcal{F}^3$ given by (\ref{Theta}) with $\phi = \psi =0$ is a Courant structure on $A \oplus A^*$ if and only if $\left((A,\mu),(A^*,\gamma)\right)$ is a Lie bialgebroid \cite{roy}.


\section{Nijenhuis concomitant of two tensors}

Let $(A \oplus A^*, \Theta)$ be a Courant algebroid and $I$  a vector bundle endomorphism of $A \oplus A^*$, $I:A \oplus A^* \to A \oplus A^*$. If
$\langle Iu,v \rangle + \langle u,Iv \rangle =0,$
for all $u,v \in A \oplus A^*$,  $I$ is said to be {\em skew-symmetric}.
Vector bundle endomorphisms of $A \oplus A^*$  will be seen as $(1,1)$-tensors on $A \oplus A^*$.

The deformation of the Dorfman bracket $[.,.]$ by a $(1,1)$-tensor $I: A \oplus A^* \to A \oplus A^*$ is the bracket $[.,.]_I$ defined, for all $\mathcal{X},\mathcal{Y} \in \Gamma(A \oplus A^*)$, by
$$[\mathcal{X},\mathcal{Y}]_I =[I \mathcal{X},\mathcal{Y}]+[\mathcal{X},I \mathcal{Y}]-I[\mathcal{X},\mathcal{Y}].$$

When $I$ is skew-symmetric, the deformed structure $(\rho \circ I, [.,.]_I)$ is given, in supergeometric terms, by $\Theta_{I}:=\{I,\Theta\}\in\mathcal{F}^{3}$.
The deformation of $\Theta_I$ by the skew-symmetric $(1,1)$-tensor $J$ is denoted by
$\Theta_{I,J}$, i.e., $\Theta_{I,J}=\{J,\{I, \Theta \}\},$ while
the deformed Dorfman bracket associated to $\Theta_{I,J}$  is denoted by $[.,.]_{I,J}$.

Recall that a vector bundle endomorphism $I:A \oplus A^* \to A \oplus A^*$ is a {\em Nijenhuis} tensor on the Courant algebroid $(A \oplus A^*, \Theta)$ if
its torsion vanishes. The {\em torsion} ${\mathcal T}_{\Theta}I$ is defined, for all $\mathcal{X},\mathcal{Y} \in \Gamma (A \oplus A^*)$, by

\begin{equation*} \label{torsion}
{\mathcal T}_{\Theta}I (\mathcal{X},\mathcal{Y})= [I\mathcal{X},I\mathcal{Y}]-I[\mathcal{X},\mathcal{Y}]_{I}
\end{equation*}
or, equivalently, by
\begin{equation} \label{second_def_torsion}
{\mathcal T}_{\Theta}I (\mathcal{X},\mathcal{Y})=\frac{1}{2}\big([\mathcal{X},\mathcal{Y}]_{I,I}-[\mathcal{X},\mathcal{Y}]_{I^2}\big),
\end{equation}
where $I^2=I \circ I$.
When $I^2= \lambda\, id_{A \oplus A^*}$, for some $\lambda \in \Rr$, (\ref{second_def_torsion}) is given, in supergeometric terms, by
\begin{equation} \label{supergeometric_torsion}
{\mathcal T}_{\Theta}I= \frac{1}{2}(\Theta_{I,I}-\lambda \Theta)
\end{equation}
(see \cite{grab}).

The notion of {\em Nijenhuis concomitant} of two tensor fields of type $(1,1)$ on a manifold was introduced in \cite{nijenhuis}. In the case of $(1,1)$-tensors $I$ and $J$ on a Courant algebroid $(A \oplus A^*,\Theta)$,
  the {\em Nijenhuis concomitant} of $I$
and $J$ is the map ${\mathcal{N}}_\Theta (I,J): \Gamma(A \oplus A^*) \times \Gamma (A \oplus A^*) \to \Gamma (A \oplus A^*) $  (in general not a tensor) defined, for all sections $\mathcal{X}$ and $\mathcal{Y}$ of $A \oplus A^*$, as follows:
\begin{eqnarray} \label{Nijenhuisconcomitant}
{\mathcal{N}}_\Theta (I,J)(\mathcal{X},\mathcal{Y})&=& [I\mathcal{X},J\mathcal{Y}]-I[\mathcal{X},J\mathcal{Y}]-I[J\mathcal{X},\mathcal{Y}]+I\circ J[\mathcal{X},\mathcal{Y}] \nonumber \\
& &+[J\mathcal{X},I\mathcal{Y}]-J[\mathcal{X},I\mathcal{Y}]-J[I\mathcal{X},\mathcal{Y}]+J\circ I[\mathcal{X},\mathcal{Y}],
\end{eqnarray}
where  $[. , .]$ is the Dorfman bracket corresponding to $\Theta$. Equivalently,
\begin{equation} \label{Nijenhuisconcomitant2}
{\mathcal{N}}_\Theta (I,J)(\mathcal{X},\mathcal{Y})= \frac{1}{2} \big( [\mathcal{X},\mathcal{Y}]_{I,J} + [\mathcal{X},\mathcal{Y}]_{J,I} - [\mathcal{X},\mathcal{Y}]_{I\circ J}- [\mathcal{X},\mathcal{Y}]_{J\circ I}\big).
\end{equation}
Notice that
\begin{equation} \label{N_and_T}
{\mathcal N}_\Theta(I,I)=2 {\mathcal T}_\Theta I
 \end{equation}
 while if $I$ and $J$ anti-commute, i.e., $I\circ J =-J \circ I$, then
 \begin{equation} \label{N_and_T_anticommute}
 {\mathcal N}_\Theta(I,J)(\mathcal{X},\mathcal{Y})=\frac{1}{2}( [\mathcal{X},\mathcal{Y}]_{I,J} + [\mathcal{X},\mathcal{Y}]_{J,I}).
 \end{equation}

 For any $(1,1)$-tensors $I$ and $J$ on $(A \oplus A^*,\Theta)$, we have \cite{anutuneslaurentndc}
 \begin{equation} \label{torsion_sum}
 {\mathcal T}_\Theta(I+J)={\mathcal T}_\Theta I + {\mathcal T}_\Theta J + {\mathcal N}_\Theta(I,J).
 \end{equation}

The concomitant $C_\Theta(I,J)$ of two skew-symmetric $(1,1)$-tensors $I$ and $J$ on a Courant algebroid $(A \oplus A^*,\Theta)$ is given by \cite{anutuneslaurentndc}:
\begin{equation}  \label{def_conc}
C_\Theta(I,J)=\Theta _{I,J}+\Theta_{J,I}.
\end{equation}
In other words,
\begin{equation} \label{def_conc1}
C_\Theta(I,J)(\mathcal{X},\mathcal{Y})= [\mathcal{X},\mathcal{Y}]_{I,J}+ [\mathcal{X},\mathcal{Y}]_{J,I}
 \end{equation}
for all $\mathcal{X},\mathcal{Y} \in \Gamma(A \oplus A^*)$. Combining (\ref{N_and_T_anticommute}) and (\ref{def_conc1})  we find that, in the case where $I$ and $J$ anti-commute,
\begin{equation}  \label{N_concomitant}
{\mathcal{N}}_\Theta (I,J)(\mathcal{X},\mathcal{Y})= \frac{1}{2} C_\Theta(I,J)(\mathcal{X},\mathcal{Y}),
  \end{equation}
for all  $\mathcal{X} , \mathcal{Y} \in \Gamma(A \oplus A^*)$.

\

The notion of Nijenhuis concomitant of two $(1,1)$-tensors on a Lie algebroid can also be considered. If $(A,\mu)$ is a Lie algebroid and $I,J$ are $(1,1)$-tensors
 on $A$,
${\mathcal{N}}_\mu (I,J)$ is given by (\ref{Nijenhuisconcomitant}), adapted in the obvious way. Equations (\ref{Nijenhuisconcomitant2}), (\ref{N_and_T}),
(\ref{N_and_T_anticommute}) and (\ref{N_concomitant}) also hold in the Lie algebroid case.

As in the case of Courant algebroids, for a Lie algebroid $(A, \mu)$, we use the following notation: $\mu_I= \{I, \mu \}$, if $I$ is either a bivector, a $2$-form
or a $(1,1)$-tensor on $A$.


\section{Tensors on Lie algebroids}

Let $(A,\mu)$ be a Lie algebroid and consider a $(1,1)$-tensor $N$, a bivector $\pi$ and a $2$-form $\omega$ on $A$. Associated with
 $N$, $id:= id_A$, $id^*\!:= id_{A^{*}}$, $\pi$ and $\omega$, we consider the skew-symmetric $(1,1)$-tensors on  $A\oplus A^*$, $J_N$, $J_{id}$, $J_\omega$ and $J_\pi$ given,
 in matrix form, respectively by

\

\begin{center}
$J_N= \left(
\begin{array}{cc}
N & 0\\
0 & -N^*
\end{array}
\right),
\hspace{1cm}
 J_{id}= \left(
\begin{array}{cc}
id & 0\\
0 & -id^*
\end{array}
\right)$,
\end{center}

\begin{center}
$J_{\omega}= \left(
\begin{array}{cc}
0 & 0\\
\omega^\flat & 0
\end{array}
\right)
\hspace{.7cm} \textrm{and} \hspace{.7cm}
J_{\pi}= \left(
\begin{array}{cc}
0 & \pi^\#\\
0 & 0
\end{array}
\right).$
\end{center}

\

In all the computations using the big bracket, instead of writing $J_N$, $J_{id}$, $J_{\omega}$ and $J_{\pi}$, we simply write $N$, $id$, $\omega$ and $\pi$. We use the  $(1,1)$-tensors on  $A\oplus A^*$ above to express the properties of
$N$ being Nijenhuis, $\pi$ Poisson and $\omega$ closed on the Lie algebroid $(A,\mu)$.

\begin{prop}[\cite{YKS11}]  \label{Nijenhuis}
Let $N$ be a $(1,1)$-tensor on $(A,\mu)$ such that $N^2= \lambda \, id_A$, for some $\lambda \in \Rr$. Then,  $N$ is a Nijenhuis tensor on the Lie algebroid $(A,\mu)$ if and only if  $J_N$ is a Nijenhuis tensor on the
 Courant algebroid $(A \oplus A^*, \mu)$.
\end{prop}

\begin{proof}
The assumption $N^2= \lambda \, id_A$ is equivalent to $J_{N}^2= \lambda \, id_{A \oplus A^*}$. In this case, the torsion of $J_N$ on $(A \oplus A^*, \mu)$
is given by (\ref{supergeometric_torsion}),
 with $\Theta=\mu$, and coincides with the torsion of $N$ on $(A, \mu)$.
\end{proof}

Let $I_\omega$ be the $(1,1)$-tensor on  $A\oplus A^*$, defined by $$I_\omega= J_\omega + J_{id}=\left(
\begin{array}{cc}
id & 0\\
\omega^\flat & -id^*
\end{array}
\right).$$

\begin{prop}  \label{mu_closed}
The 2-form $\omega$ is closed on $(A,\mu)$ if and only if  $I_{\omega}$ is a
 Nijenhuis tensor on the Courant algebroid $(A \oplus A^*, \mu)$.
\end{prop}

\begin{proof}
First, observe that $I_{\omega}^2=id_{A \oplus A^*}$. According to (\ref{supergeometric_torsion}),
we have
\begin{eqnarray*}
{\mathcal{T}}_\mu I_{\omega}& = & \frac{1}{2} \left(\{\omega + id, \{ \omega + id, \mu \} \} - \mu \right) \\
& = &  \frac{1}{2}  \left( \{\omega, \{id, \mu \} \} +\{id, \{ \omega, \mu \} \} + \{id, \{id, \mu \} \}- \mu \right)\\
& = & 2 \{\omega, \mu \},
\end{eqnarray*}
where we used, in the last equality, the formula
\begin{equation*}
\{id, u \} = (q-p)u,
\end{equation*}
for all $u \in {\mathcal F}^{(p,q)}$ \cite{YKSRubtsov}. Thus, $\omega$ is closed if and only if ${\mathcal{T}}_\mu I_{\omega}=0$.
\end{proof}

Recall that a bivector field $\pi$ on $A$ is a Poisson tensor on $(A,\mu)$ if $\mu_{\pi, \pi}=\{\pi, \{\pi, \mu \} \}=0$ or, equivalently, $[\pi, \pi ]_\mu=0$.

\begin{prop} [\cite{anutuneslaurentndc}] \label{pi_Nijenhuis}
The bivector $\pi$ is a Poisson tensor on $(A,\mu)$ if and only if  $J_{\pi}$ is a Nijenhuis tensor on the Courant algebroid $(A \oplus A^*, \mu)$.
\end{prop}
\begin{proof}
We have $J_{\pi}^2=0$ and, from (\ref{supergeometric_torsion}),
we get ${\mathcal{T}}_\mu J_{\pi}= \frac{1}{2} \{ \pi, \{\pi, \mu \} \}$.
\end{proof}

Notice that the $(1,1)$-tensors $J_\pi$ and $J_{id}$ anti-commute. Thus, from (\ref{N_concomitant}), we have
$$2 {\mathcal{N}}_\mu (J_\pi, J_{id})=C_\mu(\pi, id)=\{\pi, \{id , \mu \} \}+ \{id, \{ \pi, \mu \} \}= \mu_\pi - \mu_\pi=0.$$
Denoting by $I_\pi$ the $(1,1)$-tensor on  $A\oplus A^*$ defined by $$I_\pi= J_\pi + J_{id}=\left(
\begin{array}{cc}
id & \pi^\#\\
0 & -id^*
\end{array}
\right),$$ and taking into account the fact that
$${\mathcal{N}}_\mu (I_\pi, I_\pi)={\mathcal{N}}_\mu (J_\pi,J_\pi) + 2 {\mathcal{N}}_\mu (J_\pi,J_{id})+ {\mathcal{N}}_\mu
(J_{id}, J_{id})={\mathcal{N}}_\mu (J_\pi,J_\pi),$$
Proposition \ref{pi_Nijenhuis} admits the following equivalent formulation:

\begin{prop}
The bivector $\pi$ is a Poisson tensor on $(A,\mu)$ if and only if  $I_{\pi}$
 is a Nijenhuis tensor on the Courant algebroid $(A \oplus A^*, \mu)$.
\end{prop}

\section{Pairs of tensors on Lie algebroids}

In \cite{anutuneslaurentndc} we introduced a notion of compatibility for a pair of anti-commuting skew-symmetric $(1,1)$-tensors on a Courant algebroid.
\begin{defn} [\cite{anutuneslaurentndc}]
A pair $(I,J)$ of
skew-symmetric $(1,1)$-tensors on  a Courant algebroid with Courant structure $\Theta$ is said to be a {\em compatible pair},
if $I$ and $J$ anti-commute and $C_\Theta(I,J)=0$.
\end{defn}

In this section we show that well known structures defined by pairs of tensors on a Lie algebroid $(A,\mu)$, can be seen either as compatible pairs, or
 as Nijenhuis tensors
on the  Courant algebroid $(A \oplus A^*, \mu)$.

Let $(A,\mu)$ be a Lie algebroid.
Recall that a pair $(\pi, N)$, where $\pi$ is a bivector and $N$ is a  $(1,1)$-tensor on $A$ is a {\em Poisson-Nijenhuis structure} ($PN$ structure, for short) on $(A,\mu)$ if
 \begin{equation}
[ \pi, \pi ]_\mu =0, \,\,\,\, {\mathcal T}_\mu N=0, \,\,\,\, N \circ \pi^\#= \pi^\# \circ  N^*\,\,\,\, {\textrm{and}} \,\,\,\, C_\mu (\pi, N)=0.
\end{equation}

A pair $(\omega,N)$ formed by a $2$-form $\omega$ and a $(1,1)$-tensor $N$ on $A$ is an $\Omega N${\em structure} on $(A,\mu)$ if
\begin{equation}  \label{Omega_N}
d_\mu \omega=0, \,\,\,\, {\mathcal T}_\mu N=0, \,\,\,\, \omega^\flat \circ N= N^* \circ \omega^\flat \,\,\,\, {\textrm{and}} \,\,\,\, d_\mu (\omega_N)=0,
\end{equation}
where $\omega_N( .,.)=\omega(N.,.)$ or, equivalently, $\omega_N^\flat=\omega^\flat \circ N$.

A pair $(\varpi,N)$ formed by a $2$-form $\varpi$ and a $(1,1)$-tensor $N$ on $A$ is a {\em Hitchin pair} on $(A,\mu)$ if
\begin{equation}
 \varpi \textrm{ is symplectic}\footnote{A symplectic form on a Lie algebroid is a closed $2$-form which is non-degenerate (at each point).},
  \,\,\,\,  \varpi^\flat \circ N= N^* \circ \varpi^\flat \,\,\,\, {\textrm{and}} \,\,\,\, d_\mu (\varpi_N)=0.
\end{equation}

A pair $(\pi,\omega)$ formed by a bivector $\pi$ and a $2$-form $\omega$ on $A$ is a $P \Omega$ {\em structure} on $(A,\mu)$ if
\begin{equation}  \label{P_Omega}
[\pi, \pi]_\mu=0, \,\,\,\, d_\mu \omega=0 \,\,\,\, {\textrm{and}} \,\,\,\, d_\mu (\omega_{N})=0,
\end{equation}
where $N$ is the $(1,1)$-tensor on $A$ defined by $N=\pi^\# \circ \omega^\flat$.

Let us denote by $[.,.]_+$ the anti-commutator of two skew-symmetric tensors $I$ and $J$, i.e.,
$$[I,J]_+= I \circ J + J \circ I.$$

\begin{prop} \label{omegaN}
Let $(A,\mu)$ be a Lie algebroid, $N$ a Nijenhuis $(1,1)$-tensor and $\omega$ a closed $2$-form on $(A,\mu)$.
Then, the pair $(\omega, N)$ is an $\Omega N$ structure on $(A,\mu)$ if and only if $(J_{\omega},J_{N})$ is a compatible pair
on $(A \oplus A^*, \mu)$.
\end{prop}

\begin{proof}
We start by noticing that $[J_{\omega},J_{N}]_+=\left(
\begin{array}{cc}
0 & 0\\
\omega^\flat N- N^* \omega^\flat & 0
\end{array}
\right)$, so that $J_{\omega}$ and $J_{N}$ anti-commute if and only if  $\omega^\flat  N= N^*  \omega^\flat $.

Taking into account the fact that $\omega$ is closed, we have
$$C_\mu (J_{\omega},J_{N})=  \{\omega, \{N, \mu \}\} + \{N, \{\omega, \mu \}\} = \{ \omega, \{N, \mu \}\}
 =   \{\{\omega, N \}, \mu \} = 2 \{ \mu, \omega_N \},$$
where in the last equality we used $\omega_N=\frac{1}{2} \{N, \omega \}$.
Thus, the $2$-form $\omega_N$ is closed if and only if $C_\mu (J_{\omega},J_{N})=0$.
\end{proof}

In the case where $N^2 = \lambda \, id_A$, for some $\lambda \in \Rr$, we have the following characterization of an $\Omega N$ structure.

\begin{thm}  \label{NijenhuisOmegaN}
Let $(A,\mu)$ be a Lie algebroid, $\omega$ a closed $2$-form on $(A,\mu)$ and $N$ a $(1,1)$-tensor on $A$ such
that $N^2 = \lambda \, id_A$, for some $\lambda \in \Rr$. Then, the pair $(\omega, N)$ is an $\Omega N$ structure on
$(A,\mu)$ if and only if $J_{\omega}+J_{N}$ is a Nijenhuis tensor on $(A \oplus A^*, \mu)$
and $[J_{\omega},J_{N}]_+ =0$.
\end{thm}

\begin{proof}
We know from Proposition \ref{Nijenhuis} that if $N^2 = \lambda \, id_A$, for some $\lambda \in \Rr$, then ${\mathcal{T}}_\mu N=0 \, \Leftrightarrow \, {\mathcal{T}}_\mu J_N=0$.
Moreover,  ${\mathcal{T}}_\mu J_{\omega}=0$ for any $2$-form $\omega$ (see \cite{anutuneslaurentndc}).
Now, using (\ref{torsion_sum}), we have
\begin{eqnarray*}
{\mathcal{T}}_\mu (J_{\omega}+J_{N})&=& {\mathcal{T}}_\mu J_{\omega}+ {\mathcal{N}}_\mu (J_{\omega}, J_{N})+
 {\mathcal{T}}_\mu J_{N}\\
&=& {\mathcal{N}}_\mu (J_{\omega}, J_{N})+  {\mathcal{T}}_\mu J_{N}
\end{eqnarray*}
and, by counting the bi-degrees, we have that ${\mathcal{T}}_\mu (J_{\omega}+J_{N})=0$ is equivalent to
\begin{equation}  \label{conditionsI_omega}
{\mathcal{N}}_\mu (J_{\omega}, J_{N})=0 \quad {\textrm {and}} \quad {\mathcal{T}}_\mu J_{N}=0.
\end{equation}
Because $J_\omega$ and $J_N$ anti-commute, we have that ${\mathcal{N}}_\mu (J_{\omega},J_{N})= \frac{1}{2} C_\mu (J_{\omega},J_{N})$. Thus, the two conditions in (\ref{conditionsI_omega})
 mean that $\omega_N$ is closed (see the proof of Proposition \ref{omegaN}) and $N$ is Nijenhuis, respectively.
\end{proof}

For Hitchin pairs we obtain the following result:

\begin{prop}
Let $(A,\mu)$ be a Lie algebroid, $N$ a $(1,1)$-tensor on $(A,\mu)$ and $\varpi$ a  symplectic form on $(A,\mu)$.
Then, the pair $(\varpi, N)$ is a Hitchin pair on $(A,\mu)$ if and only if $(J_{\varpi},J_{N})$ is a compatible pair
 on
$(A \oplus A^*, \mu)$.
\end{prop}

\begin{rem}
There is no analogue of Theorem~\ref{NijenhuisOmegaN} for Hitchin pairs. Nevertheless,
in \cite{crainic}, the author proves that if $\varpi$ is a non-degenerate $2$-form, with inverse $\pi$, and $N$ is a $(1,1)$-tensor on $A$, the pair $(\varpi, N)$ is a Hitchin pair on $(A,\mu)$ if and only if
$J_{\sigma}+J_{N}+ J_{\pi}$ is a Nijenhuis tensor on $(A \oplus A^*, \mu)$ and $(J_{\sigma}+J_{N}+ J_{\pi})^2=- id _{A \oplus A^*}$, where $\sigma=- \varpi - \varpi \circ N^2$.
\end{rem}

In the case of PN structures, we have:
\begin{prop}
Let $(A,\mu)$ be a Lie algebroid, $N$ a Nijenhuis $(1,1)$-tensor on $(A,\mu)$ and $\pi$ a Poisson bivector on $(A,\mu)$.
Then, the pair $(\pi, N)$ is a Poisson-Nijenhuis structure on $(A,\mu)$ if and only if $(J_{\pi}, J_{N})$ is a compatible pair on $(A \oplus A^*, \mu)$.
\end{prop}

\begin{proof}
Notice that $[J_{\pi},J_{N}]_+=\left(
\begin{array}{cc}
0 & N \pi^\# - \pi^\# N^*\\
0 & 0
\end{array}
\right)$, so that $J_{\pi}$ and $J_{N}$ anti-commute if and only if  $N \pi^\# = \pi^\# N^* $.  Also, we have $C_\mu (J_{\pi},J_{N})= C_\mu (\pi, N)$.
\end{proof}

When $N^2 = \lambda \, id_A$, for some $\lambda \in \Rr$, we recover a result from \cite{YKS11}, which is a characterization of Poisson-Nijenhuis structures.

\begin{thm}
Let $(A,\mu)$ be a Lie algebroid, $\pi$ a bivector on $A$ and $N$ a $(1,1)$-tensor on $A$ such that $N^2 = \lambda \, id_A$, for some $\lambda \in \Rr$.
Then, the pair $(\pi, N)$ is a Poisson-Nijenhuis structure on $(A,\mu)$ if and only if $J_{\pi}+J_{N}$ is a Nijenhuis tensor on $(A \oplus A^*, \mu)$
and $[J_{\pi},J_{N}]_+ =0$.
\end{thm}

\begin{proof}
Using (\ref{torsion_sum}) we have,
\begin{equation*}
{\mathcal{T}}_\mu (J_{\pi}+J_{N})  =  {\mathcal{T}}_\mu J_{\pi}+ {\mathcal{N}}_\mu (J_{\pi},J_{N} )+{\mathcal{T}}_\mu J_{N}
\end{equation*}
and, by counting the bi-degrees, we get that the condition ${\mathcal{T}}_\mu (J_{\pi}+J_{N} )  =0$ is equivalent to
$${\mathcal{T}}_\mu J_{\pi}=0, \quad {\mathcal{N}}_\mu (J_{\pi},J_{N} )=0 \quad {\textrm {and}} \quad {\mathcal{T}}_\mu J_{N}=0.$$
From Proposition \ref{pi_Nijenhuis}, (\ref{N_concomitant}) and Proposition \ref{Nijenhuis}, the above equations mean that $\pi$ is a Poisson bivector,
$C_\mu(\pi, N)=0$ and $N$ is Nijenhuis, respectively.
\end{proof}

\begin{rem}
In \cite{anutuneslaurentndc} we showed that, given a bivector $\pi$ and a $(1,1)$-tensor $N$ on $A$ such that $N^2 = \lambda \, id_A$, for some $\lambda \in \Rr$,
then $(\pi, N)$ is a PN structure on $(A, \mu)$ if and only if $(J_{\pi},J_{N} )$ is a Poisson-Nijenhuis pair on the Courant algebroid $(A \oplus A^*, \mu)$.
\end{rem}

For $P \Omega$ structures, we have the following:
\begin{prop}  \label{P_OmegaA}
Let $(A,\mu)$ be a Lie algebroid, $\pi$ a Poisson bivector on $(A,\mu)$ and $\omega$ a closed $2$-form on $(A,\mu)$. Consider the $(1,1)$-tensor $N$ on $A$ defined by $N=\pi^\# \circ \omega^\flat$, and the corresponding $(1,1)$-tensor on $A\oplus A^*$, $J_N= \left(
\begin{array}{cc}
\pi^\# \circ \omega^\flat &0\\
0 & - \omega^\flat \circ \pi^\#
\end{array}
\right) $.
Then, the pair $(\pi, \omega)$ is a $P \Omega$ structure on $(A,\mu)$ if and only if  $(J_{\omega},J_N)$ is a  compatible pair on
 $(A \oplus A^*, \mu)$.
\end{prop}

\begin{proof}
It is easy to see that  $J_\omega$ and $J_N$ anti-commute. The $2$-form $\omega$ being closed we have, taking into account the fact that $J_N=\{\omega, \pi \}$,
\begin{equation} \label{pi_omega}
C_\mu (J_\omega, J_N)=C_\mu (\omega, \{ \omega, \pi \})= \{ \{  \omega, \{ \omega, \pi \} \}, \mu \}= - 2 \{\omega_N, \mu \}.
\end{equation}
So, the $2$-form $\omega_N$ is closed if and only if $C_\mu (J_\omega, J_N)=0$.
\end{proof}


\section{Exact Poisson quasi-Nijenhuis structures (with background)}
Let $(A, \mu)$ be a Lie algebroid, $H$ a closed $3$-form on  $(A, \mu)$ and consider the Courant algebroid with background $(A\oplus A^*, \mu + H)$.

Poisson quasi-Nijenhuis structures with background on Lie algebroids were introduced in \cite{antunes1}. We recall that a {\em Poisson quasi-Nijenhuis structure
 with background} on $(A, \mu)$ is a quadruple $(\pi, N, \phi, H)$, where $\pi$ is a bivector, $N$ is a $(1,1)$-tensor and $\phi$ and $H$ are closed
 $3$-forms such that $N \circ \pi^\#= \pi^\# \circ N^*$ and
 \begin{enumerate}
 \item[(i)] $\pi$ is Poisson,
 \item[(ii)] $C_\mu(\pi, N)(\alpha, \beta)= 2 H(\pi^\#(\alpha), \pi^\#(\beta),.)$, for all $\alpha, \beta \in \Gamma(A^*)$,
 \item[(iii)] ${\mathcal T}_{\mu}N (X,Y)= \pi^\# ( H(NX,Y,.)+ H(X,NY,.)+ \phi(X,Y,.))$, for all $ X,Y \in \Gamma(A)$,
 \item[(iv)] $d_{\mu_{N}} \phi= d_\mu \mathcal H$,
 \end{enumerate}
  with $\mathcal H (X,Y,Z)= \circlearrowleft_{X,Y,Z} H(NX,NY,Z)$, for all $X,Y,Z \in \Gamma(A)$, where $\circlearrowleft_{X,Y,Z}$
means sum after circular permutation on $X$, $Y$ and $Z$.

A Poisson quasi-Nijenhuis structure
 with background $(\pi, N, \phi, H)$ is called {\em exact} if $\phi= d_\mu \omega$, $\omega^\flat \circ N= N^* \circ \omega^\flat$ and condition (iv) is replaced by
\begin{itemize}
 \item[(iv')]
$
 i_N d_\mu \omega- d_\mu \omega_N- {\mathcal H}$ is proportional to $H$,
 \end{itemize}
where
 $i_N d_\mu \omega(X,Y,Z)=d_\mu \omega(NX,Y,Z)+ d_\mu \omega(X,NY,Z)+d_\mu \omega(X,Y,NZ)$, for all $X,Y,Z \in \Gamma(A)$.

In \cite{antunes1} it is proved \footnote{The quadruple considered in \cite{antunes1} is $(\pi, N, -d_\mu \omega, H)$ and should be $(\pi, N, d_\mu \omega, H)$.} that if $J_N + J_\pi + J_\omega$ is a Nijenhuis
tensor on $(A \oplus A^*, \mu + H)$ and satisfies $(J_N + J_\pi + J_\omega)^2= \lambda \, id_{A \oplus A^*}$, with $\lambda \in \{-1, 0, 1 \}$,
then the quadruple $(\pi, N, d_\mu \omega, H)$ is a Poisson quasi-Nijenhuis structure with background  on $(A, \mu)$.
It is easy to see that the same result holds for any $ \lambda \in \Rr$. It is worth noticing that $(J_N + J_\pi + J_\omega)^2= \lambda \, id_{A \oplus A^*}$,
$\lambda \in \Rr$,
is equivalent to the three conditions: $N \circ \pi^\#= \pi^\# \circ N^*$,
$\omega^\flat \circ N= N^* \circ \omega^\flat$ and $N^2 + \pi^\# \circ \omega^\flat = \lambda \, id_A$.
Using the notion of exact Poisson quasi-Nijenhuis structure with background, we deduce the following (see the proof of Theorem 2.5 in \cite{antunes1}):

\begin{thm}  \label{exact_PqNb}
Let $(A,\mu)$ be a Lie algebroid, $\pi$ a bivector, $\omega$ a $2$-form, $H$ a closed $3$-form and $N$ a $(1,1)$-tensor on $A$ such that
 $N \circ \pi^\#= \pi^\# \circ N^*$, $\omega^\flat \circ N= N^* \circ \omega^\flat$ and
 $N^2 + \pi^\# \circ \omega^\flat$ is proportional to  $id_A$. Then, $J_N + J_\pi + J_\omega$ is a Nijenhuis tensor
  on the Courant algebroid $(A \oplus A^*, \mu + H)$ if and only if the quadruple $(\pi, N, d_\mu \omega,H)$ is an exact
   Poisson quasi-Nijenhuis structure with background on $(A, \mu)$.
\end{thm}

Notice that in Theorem~\ref{exact_PqNb}, if $N^2 + \pi^\# \circ \omega^\flat = \lambda \, id_A$, for some $\lambda \in \Rr$, then the constant of proportionality that should be considered in (iv') is $- \lambda$, i.e., $i_N d_\mu \omega- d_\mu \omega_N- {\mathcal H} = - \lambda H$.

\

 A Poisson quasi-Nijenhuis structure on a Lie algebroid $(A, \mu)$ is a Poisson quasi-Nijenhuis structure with background, with $H=0$. This notion was introduced, on manifolds,
 in \cite{stienonxu} and then extended to Lie algebroids in \cite{cndcdn}.
 An exact Poisson quasi-Nijenhuis structure with background $H=0$ is called an {\em exact Poisson quasi-Nijenhuis structure}. In this case,
 the $3$-form $i_N d_\mu \omega$ is also exact.

Next, we consider $H=0$ and a special case where the assumption $N^2+ \pi^\# \circ \omega^\flat= \lambda \, id_A$, $\lambda \in \Rr$, in Theorem~\ref{exact_PqNb} is replaced by
$N^2= \lambda \, id_A$, $\lambda \in \Rr$.

 \begin{thm} \label{exactPqN}
Let $(A,\mu)$ be a Lie algebroid, $\pi$ a bivector, $\omega$ a $2$-form and $N$ a $(1,1)$-tensor on $A$ such that $N \circ \pi^\#= \pi^\# \circ N^*$,
$\omega^\flat \circ N= N^* \circ \omega^\flat$
and
$N^2 = \lambda \, id_A$,
for some $\lambda \in \Rr$. If  $J_N+ J_ \pi + J_\omega$ is a Nijenhuis tensor on the Courant algebroid
$(A \oplus A^*, \mu) $,
 then the triple $(\pi, N,  d_\mu \omega)$ is an exact
Poisson quasi-Nijenhuis structure on $(A, \mu)$.
\end{thm}

\begin{proof}
 We compute,
 \begin{align*}
 {\mathcal{N}}_{\mu}(J_N + J_{\pi}+ J_{\omega}, J_N + J_{\pi}+ J_{\omega}) =
  {\mathcal{N}}_\mu (J_{N},J_{N}) + 2 {\mathcal{N}}_\mu(J_N, J_\pi )  + 2 {\mathcal{N}}_\mu (J_{N},J_{\omega})\\
   + {\mathcal{N}}_\mu (J_{\pi},J_{\pi} )+
  2 {\mathcal{N}}_\mu (J_{\pi},J_{\omega} )+ {\mathcal{N}}_\mu (J_{\omega},J_{\omega} )\\
 = {\mathcal{N}}_\mu (J_{N},J_{N}) +  C_\mu(\pi, N) + 2 {\mathcal{N}}_\mu (J_{N},J_{\omega} )+ \mu_{\pi, \pi} + 2 {\mathcal{N}}_\mu(J_\pi, J_\omega )
 \end{align*}
 and, by counting the bi-degrees, we obtain that ${\mathcal{N}}_{\mu}(J_N + J_{\pi}+ J_{\omega}, J_N + J_{\pi}+ J_{\omega})=0$ if and only if
 \begin{itemize}
 \item[i)]
 $\mu_{\pi, \pi}=0,$
 \item[ii)]
 $   C_\mu(\pi, N) =0, $
\item [iii)]
$ {\mathcal{N}}_\mu (J_{N},J_{N}) =-2{\mathcal{N}}_\mu (J_{\pi},J_{\omega}),$
\item [iv)]
$ {\mathcal{N}}_\mu (J_{N},J_{\omega})=0.$
 \end{itemize}
Applying both members of iii) to any $X+ 0, Y+0 \in \Gamma(A \oplus A^*)$, we get
 \begin{equation*}
 {\mathcal{T}}_\mu N (X,Y)+0 = -{\mathcal{N}}_\mu (J_{\pi},J_{\omega}) (X+ 0, Y+0)
 \end{equation*}
 which gives, using (\ref{Nijenhuisconcomitant}),
 \begin{equation}
 {\mathcal{T}}_\mu N (X,Y) = \pi ^\#(d_\mu \omega(X,Y,.)).
 \end{equation}
 Using again (\ref{Nijenhuisconcomitant}) we compute, for any $X +\alpha, Y+ \beta \in \Gamma(A \oplus A^*)$,
 $$ {\mathcal{N}}_\mu (J_{N},J_{\omega})(X + \alpha, Y+ \beta)=(0, i_N d_\mu \omega(X,Y,.)-d_\mu \omega_N(X,Y,.)).$$
Thus, $$ {\mathcal{N}}_\mu (J_{N},J_{\omega})=0  \, \, \Leftrightarrow \, \, i_N d_\mu \omega = d_\mu \omega_N.$$
\end{proof}

Now, assume that $(\pi, N, d_\mu \omega)$ is an exact Poisson quasi-Nijenhuis structure on a Lie algebroid $(A,\mu)$, with $N^2  = \lambda \, id_A$, for some $\lambda \in \Rr$.
Then, we have that ${\mathcal{T}}_\mu N(X,Y)= \pi ^\#(d_\mu \omega(X,Y,.))$, $X, Y \in \Gamma(A)$ and, using the formula \cite{YKS11}
$$\langle {\mathcal{T}}_\mu J_{N}(X +\alpha, Y+ \beta), Z + \eta \rangle = \langle {\mathcal{T}}_\mu N(X,Y), \eta \rangle+
\langle {\mathcal{T}}_\mu N(Y,Z), \alpha \rangle + \langle {\mathcal{T}}_\mu N(Z,X), \beta \rangle,$$
for all $X +\alpha, Y+ \beta, Z + \eta \in \Gamma(A \oplus A^*)$, we obtain
$${\mathcal{T}}_\mu J_{N}(X +\alpha, Y+ \beta)= \pi ^\#(d_\mu \omega(X,Y,.)) + d_\mu \omega(Y, \pi^\#(\alpha),.)- d_\mu \omega(X, \pi^\#(\beta),.).$$
On the other hand, from (\ref{Nijenhuisconcomitant}), we get
 \begin{align*}
 {\mathcal{N}}_\mu ( J_ \pi , J_\omega)(X + \alpha,Y + \beta))= - \pi ^\#(d_\mu \omega(X,Y,.))- d_\mu \omega(Y, \pi^\#(\alpha),.)+ d_\mu \omega(X, \pi^\#(\beta),.)\\
  - {\mathcal L}_X(\omega^\flat \circ \pi^\# (\beta)) + \omega^\flat \circ \pi^\# ({\mathcal L}_X \beta) + i_Y d (\omega^\flat \circ \pi^\# (\alpha)) - \omega^\flat \circ \pi^\# (i_Y d \alpha).
\end{align*}
So, in the case where $\omega^\flat \circ \pi^\#= k\,id_{A^*}$, for some $k \in \Rr$, we get a converse of Theorem~\ref{exactPqN}, which is a particular case of Theorem~\ref{exact_PqNb}:

\begin{cor}
Let $(\pi, N, d_\mu \omega)$ be an exact
Poisson quasi-Nijenhuis structure on a Lie algebroid $(A, \mu)$ with
$N^2 = \lambda \, id_A$,
for some $\lambda \in \Rr$. If $\omega^\flat \circ \pi^\#= k\,id_{A^*}$, for some $k \in \Rr$,
 then $J_N+ J_ \pi + J_\omega$ is a Nijenhuis tensor on the Courant algebroid
$(A \oplus A^*, \mu) $.
\end{cor}


\section{Complementary forms of Poisson bivectors}

Let $\pi$ be a Poisson bivector on a Lie algebroid $(A,\mu)$.
Recall \cite{vaisman} that a $2$-form $\omega$ is said to be a {\em complementary form of} $\pi$ on $(A,\mu)$ if
\begin{equation}  \label{Vaisman}
[\omega, \omega]_{\mu_{\pi}}=0.
\end{equation}

It is well known that, when $\pi$ is a Poisson bivector on a Lie algebroid $(A,\mu)$, the pair $(A^*, \mu_\pi)$ is a Lie algebroid and therefore
 $(A^* \oplus A, \mu_\pi)$ is a Courant algebroid.

The next proposition, which is a dual version of Propostion \ref{pi_Nijenhuis}, characterizes
complementary forms as Nijenhuis tensors on $(A^* \oplus A, \mu_\pi)$.

\begin{prop} \label{compl_forms}
Let $\omega$ be a $2$-form and $\pi$ a Poisson bivector on a Lie algebroid $(A, \mu)$. Then, $\omega$ is a complementary form of $\pi$ if and only if
 $J_\omega$ is a Nijenhuis tensor on the Courant algebroid $(A^* \oplus A,\mu_\pi)$.
\end{prop}

\section{Compatibility of structures defined by pairs of tensors on Lie algebroids}

Usually, two geometric objects of the same type are said to be compatible if their sum is still an object of the same type.
In the same spirit, we introduce the next definition.

\begin{defn}  \label{compatible_struct}
Two $PN$ (respectively, $\Omega N$, $P \Omega$) structures on a Lie algebroid $(A, \mu)$ are said to be {\em compatible} if their sum is a $PN$
 (respectively, $\Omega N$, $P \Omega$) structure on $(A, \mu)$. Also, two Hitchin pairs on $(A, \mu)$ are said to be {\em compatible} if their sum is a Hitchin pair on $(A, \mu)$.
\end{defn}

\begin{exs}
Let $(g, N_1, N_2, N_3)$ be a hyperk\"{a}hler structure on a Lie algebroid $(A, \mu)$ \cite{bcg}. Define the K\"{a}hler forms $\omega_j$, $j=1,2,3$, by setting
$\omega_j(X,Y)=g(N_j(X),Y)$, for all $X, Y \in \Gamma(A)$.
It is known that the K\"{a}hler forms $\omega_j$ are symplectic and, if we denote by $\pi_j$, $j=1,2,3$, their inverse Poisson bivectors, we have \cite{antunesndc}:
\begin{itemize}
    \item the $\Omega N$ structures $(\omega_j, I_{j-1})$ and $(\omega_j, I_{j+1})$ are compatible;
    \item the $\Omega N$ structures $(\omega_{j-1}, I_{j})$ and $(\omega_{j+1}, I_{j})$ are compatible;
    \item the $PN$ structures $(\pi_j, I_{j-1})$ and $(\pi_j, I_{j+1})$ are compatible;
    \item the $PN$ structures $(\pi_{j-1}, I_{j})$ and $(\pi_{j+1}, I_{j})$ are compatible;
    \item the $P \Omega$ structures $(\pi_{j-1}, \omega_{j})$ and $(\pi_{j+1}, \omega_{j})$ are compatible,
\end{itemize}
where all the indices are taken in $\mathbb{Z}_3$.
\end{exs}

Next, we show that the compatibility of  $\Omega N$,  $PN$ and $P \Omega$ structures on a Lie algebroid $(A, \mu)$ can be established using the corresponding
associated tensors on the Courant algebroid $(A \oplus A^*, \mu)$.

\begin{prop}\label{compatible_sum_omegaN}
Let $(\omega, N)$ and $(\omega', N')$ be two $\Omega N$ structures on a Lie
 algebroid $(A, \mu)$. Then, $(\omega, N)$ and $(\omega', N')$ are compatible if and only if
 ${\mathcal{N}}_\mu(N,N')=0$,  ${\mathcal{N}}_\mu( J_\omega, J_{N'})+{\mathcal{N}}_\mu( J_{\omega'}, J_N)=0$ and
 $[J_{\omega}, J_{N'}]_+ + [J_{ \omega'}, J_{N}]_+ =0$.
\end{prop}

\begin{proof}
Since $N$ and $N'$ are Nijenhuis, ${\mathcal{N}}_\mu(N,N')=0$ is equivalent to $N+N'$ being Nijenhuis (see (\ref{torsion_sum})). Using Proposition~\ref{omegaN}, we need to show that $(J_{\omega + \omega'}, J_{N+ N'})$ is a compatible pair on $(A \oplus A^*, \mu)$. We have,
\begin{eqnarray*}
[J_{\omega + \omega'}, J_{N+ N'}]_+ =0 & \Leftrightarrow & (\omega + \omega')^\flat \circ (N+N') = (N+N')^{*} \circ (\omega + \omega')^\flat \\
& \Leftrightarrow & \omega^\flat \circ N'+  (\omega')^\flat \circ N =(N')^{*} \circ \omega^\flat + N^{*} \circ (\omega')^\flat \\
& \Leftrightarrow & [J_{\omega}, J_{N'}]_+ + [J_{ \omega'}, J_{N}]_+ =0.
\end{eqnarray*}
From the bilinearity of ${\mathcal{N}}_\mu$ and Proposition \ref{omegaN}, we get
\begin{equation*}
{\mathcal{N}}_\mu(J_{\omega + \omega'}, J_{N+ N'})= {\mathcal{N}}_\mu(J_{\omega}, J_{N'})+ {\mathcal{N}}_\mu(J_{\omega'}, J_{N})
\end{equation*}
and so ${\mathcal{N}}_\mu(J_{\omega + \omega'}, J_{N+ N'})=0$ if and only if ${\mathcal{N}}_\mu(J_{\omega}, J_{N'})+ {\mathcal{N}}_\mu(J_{\omega'}, J_{N})=0$.
\end{proof}

As a consequence of the previous proposition, we have:
\begin{cor}
Let $(\omega, N)$ and $(\omega', N')$ be two compatible $\Omega N$ structures on $(A, \mu)$. Then,  $(\omega, N')$ is an $\Omega N$ structure if and only if $(\omega', N)$ is an $\Omega N$ structure.
When one of the pairs, $(\omega, N')$  or $(\omega', N)$, is an $\Omega N$ structure, the four $\Omega N$ structures $(\omega, N)$,
 $(\omega', N')$, $(\omega, N')$  and $(\omega', N)$ are pairwise compatible.
\end{cor}

\begin{proof}
The $\Omega N$ structures $(\omega, N)$ and $(\omega', N')$ being compatible,
\begin{equation} \label{sum_compatible}
\omega^\flat \circ N'+ (\omega')^\flat \circ N = N'^{*} \circ \omega^\flat + N^* \circ (\omega')^\flat
\end{equation}
 holds. So, if $(\omega, N')$ (respectively, $(\omega', N)$) is an $\Omega N$ structure, then $\omega^\flat \circ N'= {N'}^* \circ \omega^\flat$ (respectively, $(\omega')^\flat \circ N= N^* \circ (\omega')^\flat$) and (\ref{sum_compatible}) implies $(\omega')^\flat \circ N = N^* \circ (\omega')^\flat$ (respectively, $\omega^\flat \circ N'= {N'}^* \circ \omega^\flat$).

 From Proposition \ref{compatible_sum_omegaN}, we have ${\mathcal{N}}_\mu( J_\omega, J_{N'})+{\mathcal{N}}_\mu( J_{\omega'}, J_N)=0$.
 Applying Proposition \ref{omegaN}, the first statement is proved. The second statement is obvious.
\end{proof}

Adapting the proof of Proposition~\ref{compatible_sum_omegaN}, we may establish the following for Hitchin pairs:

\begin{prop}
Let $(\varpi, N)$ and $(\varpi', N')$ be two Hitchin pairs on a Lie
 algebroid $(A, \mu)$ such that $\varpi + \varpi'$ is non-degenerate.
 Then, $(\varpi, N)$ and $(\varpi', N')$ are compatible if and only if
  ${\mathcal{N}}_\mu( J_\varpi, J_{N'})+{\mathcal{N}}_\mu( J_{\varpi'}, J_N)=0$ and $[J_{\varpi}, J_{N'}]_+ + [J_{ \varpi'}, J_{N}]_+ =0$.
\end{prop}

Notice that if $\pi^{\#} \circ (\varpi')^{\flat}=- (\pi')^{\#} \circ \varpi^{\flat}$, where $\pi$ and $\pi'$ are the inverses of $\varpi$ and $\varpi'$,
respectively, then the sum $\varpi + \varpi'$ is non-degenerate.

\begin{prop}
Let $(\pi, N)$ and $(\pi', N')$ be two $P N$ structures on a Lie algebroid $(A, \mu)$. Then, $(\pi, N)$ and $(\pi', N')$ are compatible if and only if
 ${\mathcal{N}}_\mu(N,N')=0$, ${\mathcal{N}}_\mu(J_\pi,J_{\pi'})=0$,
  ${\mathcal{N}}_\mu( J_\pi, J_{N'})+{\mathcal{N}}_\mu( J_{\pi'}, J_N)=0$ and $[J_{\pi}, J_{N'}]_+ + [J_{ \pi'}, J_{N}]_+ =0$.
\end{prop}

\begin{proof}
We have ${\mathcal{N}}_\mu(J_\pi,J_{\pi'})=0 \Leftrightarrow [\pi, \pi']_\mu=0$, which means that $\pi + \pi'$ is a Poisson bivector on $(A,\mu)$. As we already noticed,
${\mathcal{N}}_\mu(N,N')=0$ is equivalent to $N+N'$ being a Nijenhuis tensor on $(A,\mu)$. The condition $[J_{\pi}, J_{N'}]_+ + [J_{ \pi'}, J_{N}]_+ =0$ means that
$J_{N+N'}$ anti-commutes with $J_{\pi+\pi'}$, so that ${\mathcal{N}}_\mu( J_\pi, J_{N'})+{\mathcal{N}}_\mu( J_{\pi'}, J_N)=0$ is equivalent to $C_\mu(\pi+\pi',
N+N')=0$.
\end{proof}

As in the case of the $\Omega N$ structures, we have the following:
\begin{cor}
Let $(\pi, N)$ and $(\pi', N')$ be two compatible $PN$ structures on $(A, \mu)$. Then,  $(\pi, N')$ is a $PN$ structure if
and only if $(\pi', N)$ is a $PN$ structure. When one of the pairs, $(\pi, N')$  or $(\pi', N)$, is a $PN$ structure, the four $PN$ structures $(\pi, N)$,
 $(\pi', N')$, $(\pi, N')$  and $(\pi', N)$ are pairwise compatible.
\end{cor}

Next, we consider the compatibility of $P \Omega$ structures.

\begin{prop}
Let $(\pi, \omega)$ and $(\pi', \omega')$ be two $P \Omega$ structures on a Lie algebroid $(A, \mu)$ such that
$[J_{\pi}, J_{\omega'}]_+ + [J_{ \pi'}, J_{\omega}]_+ =0$.
Then, $(\pi, \omega)$ and $(\pi', \omega')$ are compatible if and only if ${\mathcal{N}}_\mu(J_\pi,J_{\pi'})=0$ and $C_\mu (J_{\omega}, J_{N'}) + C_\mu (J_{\omega'}, J_{N})=0$, where
$N=\pi^{\#} \circ \omega^{\flat}$ and $N'=(\pi')^{\#} \circ (\omega')^{\flat}$.
\end{prop}

\begin{proof}
First, we notice that $[J_{\pi}, J_{\omega'}]_+ + [J_{ \pi'}, J_{\omega}]_+ =0$ is equivalent to $\pi^{\#} \circ (\omega')^{\flat}=- (\pi')^{\#} \circ \omega^{\flat}$. From this, we get
$$(\pi + \pi')^\# \circ (\omega + \omega')^\flat= N + N'$$
and also
$$
(\omega + \omega')^\flat \circ (N+N')= (N+N')^* \circ (\omega + \omega')^\flat,
$$
which is equivalent to $[J_{\omega + \omega'}, J_{N+ N'}]_+=0$.
Now, applying Proposition~\ref{P_OmegaA}, we have
$$C_\mu (J_{\omega + \omega'}, J_{N+N'})= C_\mu (J_{\omega}, J_{N'}) + C_\mu (J_{\omega'}, J_{N})$$
and the proof is complete.
\end{proof}

\begin{cor} \label{cor_plicas}
Let $(\pi, \omega)$ and $(\pi', \omega')$ be two $P \Omega$ structures on a Lie algebroid $(A, \mu)$
such that $[J_{\pi}, J_{\omega'}]_+ + [J_{ \pi'}, J_{\omega}]_+ =0$. Assume that $(\pi, \omega)$ and $(\pi', \omega')$ are compatible.
 Then, $(\pi, \omega')$ is a $P \Omega$ structure if and only if $(\pi', \omega)$ is a $P \Omega$ structure. When one of the pairs, $(\pi, \omega')$  or $(\pi', \omega)$, is a $P \Omega$ structure, the four $P \Omega$ structures $(\pi, \omega)$,
 $(\pi', \omega')$, $(\pi, \omega')$  and $(\pi', \omega)$ are pairwise compatible.
\end{cor}

\begin{proof}
Let us set $\hat{N}= \pi^{\#} \circ (\omega')^{\flat}= -(\pi')^{\#} \circ (\omega)^{\flat}$ and, as before, $N=\pi^{\#} \circ \omega^{\flat}$ and $N'=(\pi')^{\#} \circ (\omega')^{\flat}$.
According to the previous proposition, $C_\mu (J_{\omega}, J_{N'}) + C_\mu (J_{\omega'}, J_{N})=0$ or,
equivalently,
\begin{equation} \label{conc}
 \{\mu, \{N' , \omega \} + \{N , \omega ' \} \}=0
\end{equation}
holds.
If $(\pi, \omega')$ is a $P \Omega$ structure, then $$\omega'_{\hat{N}}=\omega'_{\pi^{\#} \circ (\omega')^{\flat}}= \omega' \circ \pi^{\#} \circ (\omega')^{\flat}=-\omega \circ (\pi')^{\#} \circ (\omega')^{\flat}=- \omega_{N'} $$ is closed. From (\ref{conc}), we get that $\omega'_N$ is closed and since $\omega'_N= - \omega_{(\pi')^{\#} \circ \omega^{\flat}}$, the pair $(\pi', \omega)$ is a $P \Omega$ structure.
Conversely, if $(\pi', \omega)$ is a $P \Omega$ structure so is $(\pi, \omega')$.

For the second part, notice that the four pairs $(\pi, \omega)$, $(\pi', \omega')$,
$(\pi, \omega')$ and $(\pi', \omega)$ being $P \Omega$ structures, the $2$-forms
\begin{equation} \label{2forms}
\omega_N, \,\,\, \omega'_{N'}, \,\,\,\omega'_{\hat{N}}=-\omega_{N'}, \,\,\, \omega_{\hat{N}}=\omega'_N \,\,\,{\textrm{are closed}}.
\end{equation}
Also, because the $P \Omega$ structures $(\pi, \omega)$ and $(\pi', \omega')$ are compatible, we have ${\mathcal{N}}_\mu(J_\pi,J_{\pi'})=0$.

 We prove  (the other cases are similar):
 \begin{enumerate}
 \item [i)]
 $(\pi, \omega)$ and $(\pi, \omega')$ are compatible;
 \item [ii)]
 $(\pi, \omega)$ and $(\pi', \omega)$ are compatible;
 \item [iii)]
  $(\pi, \omega')$ and $(\pi', \omega)$ are compatible.
 \end{enumerate}

Case i): We  have $ \pi^{\#} \circ (\omega + \omega')^{\flat}=N + \hat{N}$ and by (\ref{2forms}), the $2$-form $(\omega + \omega')_{N + \hat{N}}$ is closed.

Case ii): In this case, $( \pi + \pi')^{\#} \circ \omega^{\flat}=N - \hat{N}$ and by (\ref{2forms}), the $2$-form $\omega_{N - \hat{N}}$ is closed.

Case iii): We  have $ (\pi + \pi')^{\#} \circ (\omega + \omega')^{\flat}=N +N'$ and by (\ref{2forms}), the $2$-form $(\omega + \omega')_{N+N'}$ is closed.
 \end{proof}

There are several interesting relations between the structures on Lie algebroids considered so far \cite{antunes2}, \cite{YKSRubtsov}.  Some of them will be useful in the sequel.

\begin{prop}  \label{relations_struct}
Let $\pi$ and $\omega$ be, respectively,  a Poisson bivector and a $2$-form on a Lie algebroid $(A,\mu)$ and consider the $(1,1)$-tensor $N= \pi^\# \circ \omega^\flat$.
\begin{enumerate}
\item[(i)]
If $(\pi, \omega)$ is a $P\Omega$ structure on $(A,\mu)$, then $(\pi, N)$ is a PN structure on $(A,\mu)$.
\item[(ii)]
If $(\pi, \omega)$ is a $P\Omega$ structure on $(A,\mu)$, then $(\omega, N)$ is an $\Omega N$ structure on $(A,\mu)$.
\item[(iii)]
The pair $(\pi, \omega)$ is a $P\Omega$ structure on $(A,\mu)$ if and only if $\omega$ is a closed complementary form of $\pi$ on $(A,\mu)$.
\end{enumerate}
\end{prop}

Under the conditions of Corollary~\ref{cor_plicas} we have, from Proposition~\ref{relations_struct}, that the pairs
\begin{itemize}
\item $(\omega, \hat{N})$, $(\omega', \hat{N})$, $(\omega, N')$ and $(\omega', N)$ are $\Omega N$ structures on $(A,\mu)$;
\item $(\pi, \hat{N})$, $(\pi', \hat{N})$, $(\pi, N')$ and $(\pi', N)$ are $PN$ structures on $(A,\mu)$.
\end{itemize}

\

Now, we treat the compatibility of complementary forms on $(A, \mu)$.
Let $\pi$ be a Poisson bivector on a Lie algebroid $(A, \mu)$ and consider the Courant algebroid $(A^* \oplus A, \mu_\pi)$.

\begin{defn}
Two complementary forms of $\pi$, $\omega$ and $\omega'$, are said to be {\em compatible} if $\omega + \omega'$ is a complementary form of $\pi$.
\end{defn}

\begin{prop}  \label{compatible_complementary}
Two complementary forms of $\pi$, $\omega$ and $\omega'$, are compatible if and only if  $ {\mathcal{N}}_{\mu_{\pi}} (J_{\omega},J_{\omega'})=0$.
\end{prop}

\begin{proof}
From Proposition \ref{compl_forms}, the complementary forms $\omega$ and $\omega'$ are compatible if and only if
\begin{equation}  \label{compl_compatible}
{\mathcal{T}}_{\mu_{\pi}} J_{\omega + \omega'}={\mathcal{T}}_{\mu_{\pi}} (J_{\omega}+J_{\omega'})= 0.
\end{equation}
Since ${\mathcal{T}}_{\mu_{\pi}} (J_{\omega}+J_{\omega'})= {\mathcal{T}}_{\mu_{\pi}} J_{\omega} + {\mathcal{T}}_{\mu_{\pi}} J_{\omega'} +
 {\mathcal{N}}_{\mu_{\pi}} (J_{\omega},J_{\omega'})$ and ${\mathcal{T}}_{\mu_{\pi}} J_{\omega}={\mathcal{T}}_{\mu_{\pi}} J_{\omega'}=0$, (\ref{compl_compatible})
 is equivalent to $ {\mathcal{N}}_{\mu_{\pi}} (J_{\omega},J_{\omega'})=0$.
\end{proof}


\section{Structures on deformed Lie algebroids}
We start by proving that  if a pair of tensors defines a certain structure on a Lie algebroid, this pair defines the same structure  for a whole hierarchy of deformed Lie algebroids.

It is well known that if $N$ is a Nijenhuis tensor on a Lie algebroid $(A,\mu)$ then $(A,\mu_N)$ is also a Lie algebroid.
When the Lie algebroid
structure $\mu$ is successively deformed by the same $(1,1)$-tensor $N$, we use the following notation:
\begin{equation}\label{notation_Theta_I_n}
\mu_{N^{[k]}}= \mu_{{\scriptsize N,\stackrel{k}{\dots}, N}}, \textrm{ for }k \geq 1, \textrm{ and }\mu_{N^{[0]}}= \mu.
\end{equation}

\begin{lem}[\cite{magriYKS}]\label{lema_YKS_N_k_hierarchy}
Let $N$ be a Nijenhuis tensor on a Lie algebroid $(A,\mu)$. Then, for all $n, k \in \Nn$,
\begin{enumerate}
  \item[(i)] the $(1,1)$-tensor $N^n$ is Nijenhuis with respect to $\mu_{N^k}$;
  \item[(ii)] the Lie algebroid structures $\mu_{N^k}$ and $\mu_{N^{[k]}}$ on $A$ coincide.
\end{enumerate}
\end{lem}

Before proceeding, we make a simple observation: if $\omega$ is a $2$-form and $T_0, T_1, \cdots, T_k$, $k \in \Nn$, are  $(1,1)$-tensors on a Lie algebroid, then
\begin{equation} \label{deformation_omega}
(((\omega \,_{T_{0}})_{T_{1}})_{\cdots})_{T_{k}}= \omega \,_{T_0 \circ T_1 \circ \cdots \circ T_k}.
\end{equation}

A direct computation gives the following:
\begin{lem}\label{lema_commut_omegaNn_Nk}
Let $\omega$ and $N$ be, respectively, a $2$-form and a $(1,1)$-tensor on $(A,\mu)$ such that $\omega^\flat \circ N= N^* \circ \omega^\flat$. Then,
 $(\omega_{N^n})^\flat\circ N^m=(N^m)^*\circ(\omega_{N^n})^\flat$, for all $m,n \in  \Nn$.
\end{lem}

As a consequence of the Lemma above, $(\omega_{N^n})_{N^m}=\omega_{N^{n+m}}$ is a $2$-form and
\begin{equation}\label{omega_Nn_Nm_in_BB}
   \omega_{N^{n+m}}=\frac{1}{2}\{N^m, \omega_{N^n}\},
\end{equation}
for all $m, n \in \Nn$.

\begin{prop} \label{closed_omegaNn}
Let $(\omega, N)$ be an  $\Omega N$ structure on a Lie algebroid $(A,\mu)$. Then, the $2$-form $\omega_{N^n}$ is closed on the Lie algebroid $(A, \mu_{N^{[k]}})$, for all
 $n, k \in \Nn$.
\end{prop}
\begin{proof}
Let $(\omega, N)$ be an  $\Omega N$ structure on $(A,\mu)$. First, we prove the statement for $k=0$, i.e., the $2$-form $\omega_{N^n}$ is closed with respect to $\mu$, for all
$n \in \Nn$. This is done by induction on $n$.
By hypothesis, the $2$-forms $\omega$ and $\omega_N$ are closed with respect to $\mu$. Let us suppose that, for some $r \in \Nn$, the $2$-forms $\omega_{N^{r-1}}$ and $\omega_{N^r}$ are closed with respect to $\mu$. Then, using (\ref{omega_Nn_Nm_in_BB}) and the Jacobi identity, we have
\begin{align*}
0&= \{\mu,\omega_{N^{r}} \} = \frac{1}{2}\{\mu,\{N,\omega_{N^{r-1}} \}\}= \frac{1}{2}\Big(\{ \{ \mu,N \}, \omega_{N^{r-1}} \}+ \{N, \{ \mu, \omega_{N^{r-1}} \}\}\Big)\\
&= -\frac{1}{2} \{\{ N, \mu \},\omega_{N^{r-1}}  \}.
\end{align*}
Applying $\{N, . \}$ to the last equation, using (\ref{omega_Nn_Nm_in_BB}) and the induction hypothesis, we get
\begin{align*}
0 &=\{ N,\{ \{ N, \mu \},\omega_{N^{r-1}} \} \}\\
&=\{ \{ N, \{ N, \mu \}\},\omega_{N^{r-1}}  \} +\{ \{ N, \mu \}, \{ N,\omega_{N^{r-1}} \} \}\\
&=\{ \{ N^2, \mu \}, \omega_{N^{r-1}} \} + 2\,\{ N,\{ \mu,\omega_{N^{r}} \} \} +2\,\{\{ N, \omega_{N^{r}} \}, \mu \}\\
&=\{ N^2,\{ \mu,\omega_{N^{r-1}} \} \} + \{\{ N^2, \omega_{N^{r-1}} \}, \mu \} + 4\, \{ \omega_{N^{r+1}},\mu \}\\
&=2\, \{ \omega_{N^{r+1}},\mu \}+ 4\, \{ \omega_{N^{r+1}},\mu \}\\
&=-6\,d_\mu(\omega_{N^{r+1}}),
\end{align*}
where we used, in the third equality, $\{ N, \{ N, \mu \}\}=\{ N^2, \mu \}$.
By induction, we conclude that $\d_\mu (\omega_{N^n})=0$, for all $n \in \Nn$.

Now, we prove the general statement, i.e., the $2$-form $\omega_{N^n}$ is closed with respect to $\mu_{N^{[k]}}$, for all $n, k \in \Nn$. This is done by induction on $k$.
For $k=0$, the statement is proved, in the first part of the proof, for all $n \in \Nn$. Let us suppose that, for some $s \in \Nn$ and for all $n \in \Nn$, the $2$-form $\omega_{N^n}$ is closed with respect to $\mu_{N^{[s]}}$. Applying the Jacobi identity, we have
\begin{align*}
    \d_{\mu_{N^{[s+1]}}} (\omega_{N^n})&=\{\{N,\mu_{N^{[s]}}\}, \omega_{N^n}\}
    =\{N,\{\mu_{N^{[s]}}, \omega_{N^n}\}\} + \{\{N,\omega_{N^n}\}, \mu_{N^{[s]}}\}\\
    &=2\,\{\omega_{N^{n+1}}, \mu_{N^{[s]}}\}
    =0,
\end{align*}
where we used twice the induction hypothesis. Therefore, for all $n \in \Nn$, the $2$-form $\omega_{N^n}$ is closed with respect to $\mu_{N^{[s+1]}}$ and this completes the proof of the general statement.
\end{proof}

\begin{thm} \label{struct_mu_k}
Let $\pi$, $\omega$ and $N$ be, respectively,  a bivector, a $2$-form and a $(1,1)$-tensor on a vector bundle $A$, and set $T=\pi^\# \circ \omega^\flat$.
\begin{enumerate}
\item[(i)]
The pair $(\pi, N)$ is a $P N$ structure on the Lie algebroid $(A,\mu)$ if and only if it is a $P N$ structure on the Lie algebroid $(A,\mu_{N^{[k]}})$, for all
$k \in \Nn$.
\item[(ii)]
The pair $(\omega, N)$ is an $\Omega N$ structure on the Lie algebroid $(A,\mu)$ if and only if it is an $\Omega N$ structure on the Lie algebroid
$(A,\mu_{N^{[k]}})$, for all $k \in \Nn$.
\item[(iii)]
The pair $(\pi, \omega)$ is a $P\Omega$ structure on the Lie algebroid $(A,\mu)$ if and only if  it is a $P\Omega$ structure on the Lie algebroid
$(A,\mu_{T^{[k]}})$, for all $k \in \Nn$.
\item[(iv)]
The $2$-form $\omega$ is a closed complementary form of $\pi$ on the Lie algebroid $(A,\mu)$ if and only if it is a closed complementary form of $\pi$
on the Lie algebroid $(A,\mu_{T^{[k]}})$, for all $k \in \Nn$.
\end{enumerate}
\end{thm}

\begin{proof}
For each of the equivalences above, we only prove one implication since the other is obvious.
\begin{enumerate}
\item[(i)] It is a result from \cite{magriYKS}.
\item[(ii)] Consider $(\omega, N)$ an $\Omega N$ structure on the Lie algebroid $(A,\mu)$. The $(1,1)$-tensor $N$, which is Nijenhuis with respect to $\mu$, is still Nijenhuis with respect to $\mu_{N^{[k]}}$ (Lemma \ref{lema_YKS_N_k_hierarchy}). Moreover, from Proposition \ref{closed_omegaNn}, both $2$-forms $\omega$ and $\omega_N$ are closed with respect to $\mu_{N^{[k]}}, \forall k \in \Nn$. Therefore $(\omega, N)$ is an $\Omega N$ structure on the Lie algebroid $(A,\mu_{N^{[k]}})$, for all $k \in \Nn$.
\item[(iii)]Consider $(\pi, \omega)$ a $P\Omega$ structure on the Lie algebroid $(A,\mu)$. Then, from Proposition \ref{relations_struct}(i), $(\pi, T)$ is a Poisson Nijenhuis structure on $(A,\mu)$, so that $\pi$ is Poisson with respect to $\mu_{T^{[k]}}$ (\cite{magriYKS}). Using Proposition \ref{relations_struct}(ii), the pair $(\omega,T)$ is an $\Omega N$ structure on $(A,\mu)$ and, from Proposition \ref{closed_omegaNn}, both $2$-forms $\omega$ and $\omega_T$ are closed with respect to $\mu_{T^{[k]}}$. Therefore $(\pi, \omega)$ is a $P\Omega$ structure on the Lie algebroid $(A,\mu_{T^{[k]}})$, for all $k \in \Nn$.
\item[(iv)] Follows directly from Proposition \ref{relations_struct}(iii) and the equivalence (iii) above.
\end{enumerate}
\end{proof}


\section{Hierarchies of $P\Omega$ structures, $\Omega N$ structures and complementary forms}
Hierarchies of Poisson-Nijenhuis structures on  Lie algebroids were studied in \cite{magriYKS}. In this section, we
show that hierarchies of  $P\Omega$ structures, $\Omega N$ structures and complementary forms can also be constructed on Lie algebroids.
Although hierarchies of Hitchin pairs can also be defined, we will not discuss this case because the assumptions  are too
restrictive.

\

The next theorem gives a hierarchy of $P\Omega$ structures on Lie algebroids.

\begin{thm}  \label{hierarchy_POmega}
Let $(A,\mu)$ be a Lie algebroid, $\pi$ a Poisson bivector and $\omega$ a $2$-form  such that $(\pi,\omega)$ is a $P\Omega$ structure on $(A,\mu)$. Set $N= \pi^\# \circ \omega^\flat$. Then, $(N^n \pi, \omega_{N^m})$ is a $P\Omega$ structure on the Lie algebroid $(A,\mu_{N^{[k]}})$, for all $m, n, k \in \Nn$.
\end{thm}

\begin{proof}
Let $(\pi,\omega)$ be a $P\Omega$ structure on $(A,\mu)$. Then, by Theorem~\ref{struct_mu_k} (iii), $(\pi,\omega)$ is a
$P\Omega$ structure on $(A,\mu_{N^{[k]}})$, for all $k \in \Nn$. From Proposition \ref{relations_struct} (ii), we know that $(\omega,N)$ is an $\Omega N$ structure on $(A,\mu_{N^{[k]}})$ and, by Proposition \ref{closed_omegaNn}, $d_{\mu_{N^{[k]}}} (\omega_{N^m})=0$, for all $m, k \in \Nn$.

We also have, from Proposition \ref{relations_struct} (i), that $(\pi,N)$ is a Poisson-Nijenhuis structure on $(A,\mu_{N^{[k]}})$. Thus, it is well known that $N^n \pi$ is a Poisson bivector on $(A,\mu_{N^{[k]}})$, for all $n,k \in \Nn$ \cite{magriYKS}. Moreover,
$$(N^n \pi)^\# \circ (\omega_{N^m})^\flat=N^n \circ \pi^\# \circ \omega^{\flat} \circ N^m=N^{n+m+1},$$
and, from (\ref{deformation_omega}), $(\omega_{N^m})_{N^{n+m+1}}=\omega_{N^{2m+n+1}}$. The Proposition \ref{closed_omegaNn} ensures that $\omega_{N^{2m+n+1}}$ is $\mu_{N^{[k]}}$-closed. Thus, $(N^n \pi, \omega_{N^m})$ is a $P\Omega$ structure on $(A,\mu_{N^{[k]}})$, for all $m, n, k \in \Nn$.
\end{proof}

In the next theorem we construct a hierarchy of $\Omega N$ structures.

\begin{thm}  \label{hierarchy_Omega_N}
Let $(A,\mu)$ be a Lie algebroid,  $\omega$ a $2$-form  and $N$ a $(1,1)$-tensor such that $(\omega, N)$ is an  $\Omega N$ structure on $(A,\mu)$. Then, $(\omega_{N^n},N^m)$ is an  $\Omega N$ structure on the Lie algebroid $(A,\mu_{N^{[k]}})$, for all $m, n, k \in \Nn$.
\end{thm}

\begin{proof}
Let $(\omega, N)$ be an  $\Omega N$ structure on $(A,\mu)$. From Lemma \ref{lema_commut_omegaNn_Nk}, we have $(N^m)^* \circ (\omega_{N^n})^\flat = (\omega_{N^n})^\flat \circ N^m$, for all $m,n \in \Nn$.
From Theorem \ref{struct_mu_k} (ii),
 $(\omega, N)$ is an  $\Omega N$ structure on $(A,\mu_{N^{[k]}})$, for all $k \in \Nn$. The $(1,1)$-tensor $N$ being Nijenhuis on $(A,\mu_{N^{[k]}})$, $N^m$ is also Nijenhuis on $(A,\mu_{N^{[k]}})$, for all $m,k \in \Nn$ (see Lemma \ref{lema_YKS_N_k_hierarchy}). Moreover, by applying Proposition \ref{closed_omegaNn} and (\ref{deformation_omega}), the proof is complete.
\end{proof}

\begin{rem}
We can obtain a hierarchy of $\Omega N$ structures on  $(A,\mu_{N^{[k]}})$ combining Theorem \ref{hierarchy_POmega} and Proposition \ref{relations_struct}. However, the hierarchy constructed in this manner is less general than the one given by Theorem \ref{hierarchy_Omega_N}, not only because the $(1,1)$-tensor $N$ comes from an initial given $P \Omega$ structure on $(A,\mu)$, but also because the procedure consists in associating to each $P \Omega$ structure $(N^n \pi, \omega_{N^m})$, the $\Omega N$ structure
$(\omega_{N^m}, N^{m+n+1})$. Since $m+n+1 > m$, for all $m, n \in \Nn$, the hierarchy of $\Omega N$ structures obtained in this way does not contain terms of type $(\omega_{N^m}, N^{r})$, with $r \leq m$.
\end{rem}

The next theorem gives a hierarchy  of closed complementary forms and follows directly from Proposition \ref{relations_struct} (iii) and Theorem \ref{hierarchy_POmega}.

\begin{thm}  \label{hier_closed_compl_forms}
Let $(A,\mu)$ be a Lie algebroid, $\pi$ a Poisson bivector and $\omega$ a closed $2$-form  on $(A,\mu)$ which
is a complementary form of $\pi$. Set $N= \pi^\# \circ \omega^\flat$. Then, $(\omega_{N^n})_{n \in \Nn}$ is a hierarchy of closed complementary forms of $N^m \pi$ on the Lie algebroid $(A,\mu_{N^{[k]}})$, for all $m, k \in \Nn$.
\end{thm}

The next result follows from Proposition~\ref{relations_struct} and Theorem~\ref{hierarchy_Omega_N}.
\begin{prop}
Let $(A,\mu)$ be a Lie algebroid, $\pi$, $\pi'$ bivectors  and $\omega$, $\omega'$  $2$-forms on $A$ such that
$\pi^{\#} \circ (\omega')^{\flat}=- (\pi')^{\#} \circ \omega^{\flat}$ and the pairs $(\pi, \omega)$, $(\pi, \omega')$, $(\pi', \omega)$ and $(\pi', \omega')$ are $P \Omega$
structures on $(A,\mu)$. Set
 $N=\pi^{\#} \circ \omega^{\flat}$, $\hat{N}= \pi^{\#} \circ (\omega')^{\flat}$ and $N'=(\pi')^{\#} \circ (\omega')^{\flat}$.
Then,  the pairs $(\omega_{I^n}, I^m)$ and $({\omega'}_{I^n}, I^m)$, with $I \in \{N, N' ,\widehat{N} \}$, are $\Omega N$ structures on $(A,\mu_{I^{[k]}})$, for all
$n,m,k \in \Nn$.
\end{prop}


\section{Compatibility and hierarchies}
Now, we shall see that there exists a compatibility relation between the elements of each hierarchy constructed in the previous section.
Let $(\pi, N)$ be a Poisson-Nijenhuis structure on a Lie algebroid $(A, \mu)$. Then, it is known \cite{magriYKS} that, for every $k,n \in \Nn$,
the pair $(N^k \pi, N^n)$ is a Poisson-Nijenhuis structure on $(A, \mu)$. In the next proposition, we show that any two pairs of this type are compatible Poisson-Nijenhuis structures, in the sense of Definition \ref{compatible_struct}.

\begin{prop}  \label{PN_hier_compatible}
Let $(\pi, N)$ be a Poisson-Nijenhuis structure on a Lie algebroid $(A, \mu)$. Then, $(N^k \pi, N^n)$ and $(N^l \pi, N^m)$ are compatible Poisson-Nijenhuis
 structures on $(A, \mu_{N^{[r]}})$, for all $k,l,m,n,r \in \Nn$.
\end{prop}
\begin{proof}
We start proving the result for $r=0$.
From \cite{magriYKS} we know that the Poisson bivectors $N^k \pi$ and $N^l \pi$ are compatible, which is equivalent to saying that $N^k \pi + N^l \pi$ is a Poisson bivector.
Also, it is obvious that $(N^n + N^m)(N^k \pi + N^l \pi)^\# = (N^k \pi + N^l \pi)^\# (N^n + N^m)^*$.

In \cite{anutuneslaurentndc} we proved that, if $N$ is a Nijenhuis tensor, then $[X,Y]_{N^n, N^m}= [X,Y]_{N^n \circ N^m}$, for all $X,Y \in \Gamma(A)$.
Thus, from (\ref{Nijenhuisconcomitant2}), we get ${\mathcal{N}}_\mu( N^n, N^m)=0$ and, using (\ref{torsion_sum}), we have
 that $N^n + N^m$ is a Nijenhuis tensor.

Finally, since $C_\mu(N^i \pi,N^j)=0$, for every $i,j \in \Nn$, and  $C_\mu(N^k \pi + N^l \pi, N^n + N^m)$ is a sum of terms of type $C_\mu(N^i \pi,N^j)$,
we obtain $C_\mu(N^k \pi + N^l \pi, N^n + N^m)=0$, which concludes the proof for $r=0$.

Now, from Theorem~\ref{struct_mu_k} (i), if $(\pi, N)$ is a PN structure on $(A, \mu)$ then it is a PN structure on $(A, \mu_{N^{[r]}})$.
Thus, the above proof can be repeated using the Lie algebroid structure $\mu_{N^{[r]}}$ instead of $\mu$.
\end{proof}

Starting with an $\Omega N$ structure $(\omega, N)$ on a Lie algebroid $(A, \mu)$, we constructed, in Theorem \ref{hierarchy_Omega_N}, a family of $\Omega N$ structures on $(A, \mu)$. Next, we prove that any two elements of that family are pairwise compatible.

\begin{prop} \label{ON_hier_compatible}
Let $(\omega, N)$ be an $\Omega N$ structure on a Lie algebroid $(A, \mu)$. Then, $(\omega_{N^k}, N^n)$ and $(\omega_{N^l}, N^m)$ are compatible
$\Omega N$ structures on $(A, \mu_{N^{[r]}})$, for all $k,l,m,n,r \in \Nn$.
\end{prop}
\begin{proof}
As in the proof of Proposition~\ref{PN_hier_compatible}, we prove the statement for $r=0$ and then the result follows from Theorem~\ref{struct_mu_k} (ii).
The $2$-form $\omega_{N^k} + \omega_{N^l}$ is obviously closed. The equality
$(\omega_{N^k} + \omega_{N^l})^\flat \circ (N^n + N^m)= (N^n + N^m)^* \circ (\omega_{N^k} + \omega_{N^l})^\flat$ holds (Lemma \ref{lema_commut_omegaNn_Nk}) and $N^n + N^m$ is Nijenhuis (see the proof of Proposition \ref{PN_hier_compatible}).
It remains to prove that the $2$-form $(\omega_{N^k} + \omega_{N^l})_{(N^n + N^m)}$ is closed. From Proposition \ref{closed_omegaNn}, we know that any $2$-form of type $\omega_{N^{i}}$, $i \in \Nn$, is closed. Since $(\omega_{N^k} + \omega_{N^l})_{(N^n + N^m)}$ can be decomposed into a sum of terms of type $\omega_{N^{i}}$, we obtain that $(\omega_{N^k} + \omega_{N^l})_{(N^n + N^m)}$ is closed.
\end{proof}

For $P \Omega$ structures we obtain the following result.

\begin{prop} \label{compatible_hier_POmega}
Let $(\pi, \omega)$ be a $P \Omega$ structure on a Lie algebroid $(A, \mu)$. Set $N= \pi^{\#} \circ \omega ^{\flat}$. Then, $(N^n \pi, \omega_{N^m})$
and $(N^l \pi, \omega_{N^k})$ are compatible $P \Omega$ structures on $(A, \mu_{N^{[r]}})$, for all $k,l,m,n,r \in \Nn$.
\end{prop}
\begin{proof}
As in the proof of Propositions \ref{PN_hier_compatible} and \ref{ON_hier_compatible}, we prove the statement for $r=0$ and then the result follows
from Theorem~\ref{struct_mu_k} (iii).
From Proposition \ref{relations_struct} (i), we know that if $(\pi, \omega)$ is a $P \Omega$ structure,
then $(\pi, N)$ is a $PN$ structure. The Proposition \ref{PN_hier_compatible} yields  that $N^n \pi + N^l \pi$ is a Poisson bivector.

The $2$-form $\omega_{N^m} + \omega_{N^k}$ is obviously closed. Combining Proposition \ref{relations_struct} (ii) and Proposition \ref{closed_omegaNn}, we have that each $2$-form of type $\omega_{N^{i}}$, $i \in \Nn$, is closed. Since $(\omega_{N^m} + \omega_{N^k})_{(N^n \pi + N^l \pi)^{\#} (\omega_{N^m} + \omega_{N^k})}$ can be decomposed into a sum of terms of type $\omega_{N^{i}}$, we obtain that $(\omega_{N^m} + \omega_{N^k})_{(N^n \pi + N^l \pi)^{\#} (\omega_{N^m} + \omega_{N^k})}$ is closed and the proof is complete.
\end{proof}

The next proposition shows that the elements of the hierarchy established in Theorem \ref{hier_closed_compl_forms} are pairwise compatible.
\begin{prop}
Let $(A, \mu)$ be a Lie algebroid, $\pi$ a Poisson bivector and $\omega$ be a closed complementary form of $\pi$. Set $N= \pi^\# \circ \omega^\flat$.
Then, $\omega_{N^n}$ and $\omega_{N^m}$ are compatible closed complementary forms of $N^k \pi$, for all $n,m, k \in \Nn$.
\end{prop}

\begin{proof}
According to Proposition \ref{compatible_complementary}, we have to prove that ${\mathcal{N}}_{\mu_{N^{k} \pi}} (J_{\omega_{N^n}},J_{\omega_{N^m}})=0$ or,
 equivalently, $\{ \omega_{N^n}, \{ \omega_{N^m} , \{ N^k \pi, \mu \} \} \}=0$.

From Proposition \ref{relations_struct} (ii), the pair $(\omega,N)$ is an $\Omega N$ structure. Applying Proposition \ref{closed_omegaNn} and
the Jacobi identity, we get
 \begin{align*}
\{ \omega_{N^n}, \{ \omega_{N^m} , \{ N^k \pi, \mu \} \} \}& =\{ \omega_{N^n}, \{ \{ \omega_{N^m} , N^k \pi \}, \mu \}\}\\
&=
\{ \omega_{N^n},\{ N^{k+m+1}, \mu \} \}=0,
\end{align*}
which completes the proof.
\end{proof}

\bigskip

\noindent {\bf Acknowledgments.} The authors wish to thank Camille Laurent-Gengoux for comments and
suggestions on a preliminary version of this manuscript. This work was partially supported by CMUC and FCT (Portugal), through
European program COMPETE/FEDER and by FCT grant PTDC/MAT/099880/2008.


\end{document}